\documentclass{amsart}
\usepackage{
  amsmath,
  amssymb,
  tikz,
  amsthm,
  thmtools,
  microtype,
  paralist,
  stmaryrd
}
\usepackage[nocompress]{cite}
\usepackage[charter]{mathdesign}
\usepackage{eucal}
\usepackage{hyperref}

\DeclareMathOperator{\br}{br}
\DeclareMathOperator{\ch}{ch}
\renewcommand{\c}{{\rm c}}
\DeclareMathOperator{\td}{td}
\DeclareMathOperator{\Bog}{Bog}

\DeclareMathOperator{\im}{im}
\newcommand{\orb}{\mathcal}

\renewcommand{\O}{{\mathcal O}}

\ifcsname theorem\endcsname{}\else\declaretheorem[parent=section]{theorem}\fi
\ifcsname corollary\endcsname{}\else\declaretheorem[sibling=theorem]{corollary}\fi
\ifcsname lemma\endcsname{}\else\fi
\ifcsname proposition\endcsname{}\else\declaretheorem[sibling=theorem]{proposition}\fi
\ifcsname conjecture\endcsname{}\else\declaretheorem[sibling=theorem]{conjecture}\fi
\ifcsname problem\endcsname{}\else\fi
\ifcsname question\endcsname{}\else\fi
\ifcsname definition\endcsname{}\else\declaretheorem[sibling=theorem, style=definition]{definition}\fi
\ifcsname exercise\endcsname{}\else\fi
\ifcsname example\endcsname{}\else\declaretheorem[sibling=theorem, style=definition]{example}\fi
\ifcsname remark\endcsname{}\declaretheorem[sibling=theorem, style=remark]{remark}\fi

\providecommand {\Z}{{\bf Z}}
\providecommand {\Q}{{\bf Q}}

\renewcommand {\P}{{\bf P}}

\providecommand{\SL}{\operatorname{SL}}

\providecommand {\from}{{\colon}}

\providecommand{\Hom}{\operatorname{Hom}}
\providecommand{\Ext}{\operatorname{Ext}}
\providecommand{\Tor}{\operatorname{Tor}}
\providecommand{\End}{\operatorname{End}}

\providecommand{\Pic}{\operatorname{Pic}}
\providecommand{\Sym}{\operatorname{Sym}}
\providecommand{\rk}{\operatorname{rk}} \declaretheorem[sibling=theorem,style=remark]{remark}
\numberwithin{equation}{section}

\title{Syzygy divisors on Hurwitz spaces}
\author{Anand Deopurkar \and Anand Patel}

\begin{document}
\maketitle

\begin{abstract}
  We describe a sequence of effective divisors on the Hurwitz space $H_{d,g}$ for $d$ dividing $g-1$ and compute their cycle classes on a partial compactification.
  These divisors arise from vector bundles of syzygies canonically associated to a branched cover.
  We find that the cycle classes are all proportional to each other.
\end{abstract}

\section{Introduction}
The Hurwitz space $H_{d,g}$ is the moduli space of maps $\alpha \from C \to \P^1$, where $C$ is a smooth algebraic curve of genus $g$ and $\alpha$ is finite of degree $d$.
It is one of the oldest moduli spaces studied in algebraic geometry.
Indeed, its idea goes back to the time of Riemann---a time when algebraic curves were thought of primarily as branched covers of the Riemann sphere.
It was put on a rigorous modern algebraic footing by Fulton \cite{ful:69}.
It was compactified by Harris and Mumford \cite{har.mum:82}, whose construction was then refined by Mochizuki \cite{moc:95} and Abramovich, Corti, and Vistoli \cite{abr.cor.vis:03}.
We refer the reader to \cite{rom.wew:06} for an introduction to Hurwitz spaces.

The Hurwitz spaces have attracted mathematical attention not only because of their intrinsic appeal, but also because of their role in illuminating the geometry of the moduli space $M_g$.
Indeed, it was through the Hurwitz spaces that Riemann \cite{rie:57} computed the dimension of $M_g$ and Severi \cite{sev:68}, building on work of Clebsch \cite{cle:73} and Hurwitz \cite{hur:01}, showed that $M_g$ is irreducible.
In more recent times, Harris and Mumford \cite{har.mum:82} used the compactified Hurwitz spaces to carry out a divisor class computation to show that $\overline M_g$ is of general type for large $g$.
Hurwitz spaces and their variants have also been of interest outside of algebraic geometry.
For example, spaces of branched covers of $\P^1$ with a given Galois group feature prominently in inverse Galois theory \cite{fri:77,fri.vol:91}, and spaces of covers of $\P^1$ by $\P^1$ play a key role in dynamics \cite{sil:98}.

Although Hurwitz spaces have been around for centuries, there are still many basic open questions about them.
One such question is the question of placing them in the birational classification of varieties:
For which $d$ and $g$ is $H_{d,g}$ rational, unirational, uniruled, rationally connected, or of general type?
As with many other questions about them, the answer is known only for very small or very large $d$.
For $d \leq 5$, the space $H_{d,g}$ is known to be unirational, thanks to a determinantal description of covers of degree up to 5.
For $d > \lfloor g/2 \rfloor$, the space $H_{d,g}$ dominates $M_g$.
Since $M_g$ has non-negative Kodaira dimension for $g \geq 22$, in this case $H_{d,g}$ cannot be uniruled.  
The intermediate cases are rather mysterious, but remain an active area of research.
See, for example, the recent work \cite{gei:12} on unirationality results for $d = 6$ and all $g \leq 28$ (and several more, up to $g = 45$).

At the heart of determining the birational type of $\overline H_{d,g}$ is the question of understanding its cone of effective divisors.
Indeed, modulo an obstruction coming from singularities, saying that $\overline H_{d,g}$ is of general type is equivalent to saying that its canonical class lies in the interior of its effective cone.
To show this, we need not know the full effective cone; it would suffice to know enough effective divisors, whose classes contain the canonical class in their convex span.

One way to get effective divisors is by using topology.
A general point of $H_{d,g}$ parametrizes simply branched covers.
In codimension 1, this simple topological picture can specialize in two ways: the cover can develop a ramification point of index 2 or can have two ramification points of index 1 over the same branch point.
The two possibilities give two effective divisors on $H_{d,g}$.

The goal of this paper is to describe a number of other effective divisors on $H_{d,g}$ (for $d$ dividing $g-1$) and to compute their classes on a partial compactification $\widetilde H_{d,g}$.
Their origin is distinctly algebraic, orthogonal to any topological considerations.
They are in the spirit of the classical Maroni divisor, and generalize the Casnati--Ekedahl divisors studied by the second author \cite{pat:15}.

Before describing the divisors, we recall the Maroni divisor.
A finite map $\alpha \from C \to \P^1$ canonically factors as an embedding $\iota \from C \to \P E$ followed by a projection $\pi \from \P E \to \P^1$, where $E$ is a vector bundle of rank $d-1$ and degree $g+d-1$ on $\P^1$.
In the cases where the rank divides the degree, the bundle $E$ is balanced for generic $\alpha$---it is a twist of the trivial bundle.
The \emph{Maroni divisor} is the locus of $\alpha$ for which it is unbalanced.

Our divisors $\mu_1, \dots, \mu_{d-3}$, which we call \emph{syzygy divisors}, are defined analogously for a sequence of other vector bundles $N_1, \dots, N_{d-3}$ associated with $\alpha$.
Roughly, $N_i$ is the bundle whose fiber at $t \in \P^1$ is the vector space of $(i-1)$th syzygies among the generators of the homogeneous ideal of $C_t \subset \P E_t$.
If $d$ divides $g-1$, then $N_i$ has rank dividing the degree, and for generic $\alpha$, it is balanced.
The divisor $\mu_i$ is the locus where it is unbalanced.

\begin{theorem}[Main]
  \label{thm:main}
  Suppose $d$ divides $g-1$.
  Let $i$ be an integer with $1 \leq i \leq d-3$.
  The locus $\mu_i \subset \widetilde H_{d,g}$ is an effective divisor whose class in $\Pic_\Q\widetilde H_{d,g}$ is given by
  \[ [\mu_i] = A_i \left(6(gd-6g+d+6) \cdot \zeta - d(d-12)\cdot \kappa - d^2 \cdot \delta \right),\]
  where
  \[ A_i = {{d-4} \choose {i-1}}^2 \frac{(d-2)(d-3)}{6(i+1)(d-i-1)}.\]
\end{theorem}
In the theorem, $\widetilde H_{d,g}$ is the coarse moduli space of $\alpha \from C \to \P^1$, where $C$ is an irreducible curve of arithmetic genus $g$ with at worst nodal singularities and $\alpha$ is a finite map of degree $d$.
The classes $\zeta$, $\kappa$, and $\delta$ are certain tautological divisor classes in $\Pic_\Q\widetilde H_{d,g}$; they are conjectured to generate $\Pic_\Q\widetilde H_{d,g}$.
See \autoref{sec:div} for definitions.

The most surprising feature of the divisor class is that (up to scaling) it is independent of $i$.
However, we do not expect the divisors $\mu_i$ themselves to be supported on the same set (see \autoref{sec:supports}). 
This is reminiscent of a similar phenomenon with the Brill--Noether divisors on $\overline M_g$---the classes of all the divisorial Brill--Noether loci are proportional, although the loci themselves are different.

\autoref{thm:main} gives the class of $\mu_i$ on a \emph{partial} compactification $\widetilde H_{d,g}$ of $H_{d,g}$.
It is an interesting (and challenging) problem to compute the class of the closure of $\mu_i$ on a full compactification.
This was carried out for the Maroni divisor for $d = 3$ in \cite{deo.pat:13} and for higher $d$ in \cite{gee.kou:17}.
It would also be interesting to find replacements for $\mu_i$ when $d$ does not divide $g-1$.
This would be analogous to the replacement of the Maroni divisor in the case of odd genus trigonal curves found in \cite{deo.pat:13}.

The paper is organized as follows.
In \autoref{sec:div}, we recall the (largely conjectural) description of the Picard group of $\widetilde H_{d,g}$ and describe a number of divisor classes on $\widetilde H_{d,g}$.
These include the syzygy divisors $\mu_1, \dots, \mu_{d-3}$, whose existence is contingent on the fact that the syzygy bundles $N_i$ are balanced for a generic cover.
In \autoref{sec:gen}, we discuss the generic splitting type of the syzygy bundles $N_i$.
In \autoref{sec:computation}, we carry out the main computation of the divisor class of $\mu_i$.

We work over an algebraically closed field of characteristic zero.
All schemes and stacks are locally of finite type over this field.
A \emph{point} is a closed point, unless mentioned otherwise.
The projectivization $\P U$ of a vector bundle $U$ denotes the space of one-dimensional \emph{quotients}.
The Hurwitz spaces $\widetilde H_{d,g}$ and $H_{d,g}$ have \emph{unparametrized} source and target.
That is, $\alpha_1 \from C_1 \to \P^1$ and $\alpha_2 \from C_2 \to \P^1$ are considered isomorphic if there are isomorphisms $\phi \from C_1 \to C_2$ and $\psi \from \P^1 \to \P^1$ with $\alpha_2 \circ \phi = \psi \circ \alpha_1$.

\section{Divisors on Hurwitz spaces}\label{sec:div}
The goal of this section is to recall the rational Picard group of the Hurwitz space and some divisor classes in it.
Since we are only interested in the rational Picard group, we may work either with the Deligne--Mumford stack $\widetilde{\orb H}_{d,g}$ or the coarse moduli space $\widetilde H_{d,g}$.
We will pass from one to the other without comment.

\subsection{The rational Picard group}
Denote by
\[ \alpha \from {\orb C} \to {\orb P}\]
the universal object over $\widetilde{\orb H}_{d,g}$.
Here $\pi \from \orb C \to \widetilde{\orb H}_{d,g}$ is a family of irreducible genus $g$ curves with at worst nodal singularities, $p \from \orb P \to \widetilde{\orb H}_{d,g}$ is a family of smooth genus $0$ curves, and $\alpha$ a finite morphism of degree $d$ over $\widetilde{\orb H}_{d,g}$.
The universal family allows us to write the following two `tautological' divisor classes on $\mathcal H_{d,g}$:
\[ \pi_*\left(\c_1(\omega_\pi^2)\right) \quad \text{and} \quad p_* \left(\c_1(\omega_\pi) \cdot \c_1(\alpha^* \omega_p)\right).\]
(The third natural product $\c_1(\omega_p)^2$ vanishes.)
Set 
\[ \kappa = \pi_*\left(\c_1(\omega_\pi^2)\right) \quad \text{and} \quad \zeta = \frac{-1}{2}p_* \left(\c_1(\omega_\pi) \cdot \c_1(\alpha^* \omega_p)\right).\]
Together with the divisor $\Delta$---the locus of $\alpha \from C \to \P^1$ where $C$ is singular---we get three divisor classes on $\widetilde H_{d,g}$.
Conjecturally, these exhaust the Picard group, at least modulo torsion.
\begin{conjecture}[See \cite{dia.edi:96}]
  \label{conj:pic}
  The divisor classes $\kappa$, $\xi$, and $\Delta$ generate $\Pic_\Q(\widetilde H_{d,g})$.
\end{conjecture}

The conjecture has been proved for $d \leq 5$ and for $d > 2g-2$.
For $d \leq 5$, the proof uses the unirational parametrization of $\widetilde H_{d,g}$ \cite{deo.pat:15}.
For $d > 2g-2$, the proof uses the fibration $\widetilde H_{d,g} \to M_g$ and the deep result of Harer that $\Pic_Q M_g$ has rank 1 \cite{moc:95,dia.edi:96}.
The intermediate cases are still open.
At any rate, all the divisors we consider in this paper can be written explicitly as linear combinations of $\kappa$, $\zeta$, and $\Delta$.

\subsection{Divisors from the topology of covers}\label{sec:ram}
We have three natural divisors on $\widetilde H_{d,g}$ arising from topological considerations.
A generic point of $\widetilde H_{d,g}$ represents a cover $\alpha \from C \to \P^1$ that has simple branching.
That is, $\alpha$ has $b = 2g+2d-2$ distinct branch points and over each branch point, there is a unique ramification point at which the local degree of $\alpha$ is $2$.

A simply branched cover specializes in three topologically distinct ways in codimension 1; each possibility gives a divisor on $\widetilde H_{d,g}$.
The divisor $T$ is the locus of $\alpha$ that have a point of higher ramification---a point $x \in C$ at which the local degree of $\alpha$ is at least $3$.
The divisor $D$ is the locus of $\alpha$ that have at least two distinct ramification points over the same branch point.
The divisor $\delta$ is the locus of $\alpha$ whose domain $C$ is singular. 
It is easy to see that $T$, $D$, and $\delta$ are irreducible divisors in $\widetilde H_{d,g}$.

\begin{remark}\label{rem:top}
  We can use the topological considerations above to obtain locally closed subsets of $\widetilde H_{d,g}$ of higher codimension.
  Doing so gives a stratification of $\widetilde
  H_{d,g}$ according to the topological type of $\alpha \from C \to
  \P^1$.  A complete specification of the topological type of $\alpha$
  is rather intricate.  It includes, for example, the types of
  ramification profiles for $\alpha$, the number of singularities of
  $C$, and the location of the singularities relative to the
  ramification profiles.
\end{remark}

\subsection{Divisors from the algebra of covers}\label{sec:splitting}
Just as we get special loci in $\widetilde H_{d,g}$ from non-generic topological behavior, we get special loci in $\widetilde H_{d,g}$ from non-generic algebraic behavior.
We make this precise using a structure theorem for finite morphisms due to Casnati and Ekedahl \cite{cas.eke:96}, which we first recall.

Let $X$ and $Y$ be integral schemes and $\alpha \from X \to Y$ a finite flat Gorenstein morphism of degree $d \geq 3$.
The map $\alpha$ gives an exact sequence
\begin{equation}\label{structure sheaf sequence}
  0 \to \O_Y \to \alpha_* \O_X \to {E_\alpha}^\vee \to 0,
\end{equation}
where $E = E_\alpha$ is a vector bundle of rank $(d-1)$ on $Y$, called the \emph{Tschirnhausen bundle} of $\alpha$.
Denote by $\omega_\alpha$ the dualizing sheaf of $\alpha$.
Applying $\Hom_Y(-,\O_Y)$ to \eqref{structure sheaf sequence}, we get
\begin{equation}\label{dual structure sheaf sequence}
  0 \to E \to \alpha_* \omega_\alpha \to \O_Y \to 0.
\end{equation}
The map $E \to \alpha_* \omega_\alpha$ induces a map $\alpha^* E \to \omega_\alpha$.
\begin{theorem}
  [See {\cite[Theorem~2.1]{cas.eke:96}}]
  \label{thm:CE}
  In the above setup, $\alpha^* E \to \omega_\alpha$ gives an embedding $\iota \from X \to \P E$ with $\alpha  = \pi \circ \iota$, where $\pi \from \P E \to Y$ is the projection.
  Moreover, the following hold.
  \begin{enumerate}
  \item The resolution of $\O_X$ as an $\O_{\P E}$ module has the form
    \begin{equation}\label{eqn:casnati_resolution}
      \begin{split}
        0 \to \pi^* N_{d-2} (-d) \to \pi^* N_{d-3}(-d+2) \to \pi^*N_{d-4}(-d+3) \to \dots \\
        \dots \to \pi^*N_2(-3) \to \pi^*N_1(-2) \to \O_{\P E} \to \O_X \to 0,
      \end{split}
    \end{equation}
    where the $N_i$ are vector bundles on $Y$.
    Restricted to a point $y \in Y$, this sequence is the minimal free resolution of $X_y \subset \P E_y$.
  \item The ranks of the $N_i$ are given by
    \[ \rk N_i = \frac{i(d-2-i)}{d-1} {d \choose {i+1}},\]
  \item We have $N_{d-2} \cong \pi^* \det E$.
    Furthermore, the resolution is symmetric, that is, isomorphic to the resolution obtained by applying $\Hom_{\O_{\P E}}(-, N_{d-2}(-d))$.
  \end{enumerate}
\end{theorem}
We call the resolution in \eqref{eqn:casnati_resolution} the \emph{Casnati--Ekedahl} resolution of $\alpha$.

Let us take $Y = \P^1$.
Every vector bundle on $\P^1$ splits as a direct sum of line bundles.
The multi-set of degrees of the line bundles appearing in the direct sum decomposition is unique.
We refer to this multi-set as the \emph{splitting type} of the bundle.
We say that a bundle $V$ is \emph{balanced} if the splitting type is $\{a,\dots, a\}$ for some $a$.

\begin{proposition}\label{prop:ni}
  Let $\alpha \from C \to \P^1$ be a point of $\widetilde H_{d,g}$.
  Denote by $E$ the Tschirnhausen bundle and by $N_i$ the syzygy bundles in the Casnati--Ekedahl resolution of $\alpha$.
  Then
  \begin{align*}
    \deg E &= (g+d-1), \text{ and }\\
    \deg N_i &= (d-2-i) (g+d-1) {{d-2} \choose i-1}.
  \end{align*}
\end{proposition}
\begin{proof}
  The branch divisor of $\alpha$ is cut out by a section of $(\det E)^{\otimes 2}$.
  Therefore, we get $2 \deg E = 2g+2d-2$, from which the first equation follows.
  We postpone the proof of the second equation to \autoref{sec:computation} (See \autoref{cor:degNi}). 
\end{proof}

Suppose $d$ divides $g-1$.
Then the rank of $N_i$ divides its degree.
\begin{proposition}\label{prop:gen_balanced}
  If $d$ divides $g-1$, then for a generic $\alpha \from C \to \P^1$ in $\widetilde H_{d,g}$ and $i = 1, \dots, d-2$, the bundle $N_i$ is balanced.
\end{proposition}
We postpone the proof to \autoref{sec:gen}.

\begin{definition}\label{def:main}
  Suppose $d$ divides $g-1$.
  Define the \emph{$i$th syzygy divisor} $\mu_i \subset \widetilde H_{d,g}$ as the locus of $\alpha \from C \to \P^1$ for which the bundle $N_i$ is unbalanced.
\end{definition}

There is a natural scheme structure on $\mu_i \subset \widetilde {\orb H}_{d,g}$, defined as follows.
Let $U \to \widetilde{\orb H}_{d,g}$ be an \'etale local chart for the moduli stack over which the conic bundle $\orb P_U \to U$ admits a relative $\O(1)$.
Consider the bundle $\End(N_i) \otimes \O(-1)$ on $\orb P_U$.
Note that $\chi \left(\End(N_i) \otimes \O(-1)\right) = 0$ and $h^1(\End(N_i) \otimes \O(-1) \geq 1$ if and only if $N_i$ is unbalanced.
The divisor $\mu_i$ is the zero locus of the first Fitting ideal of $R^1p_* \left(\End(N_i) \otimes \O(-1) \right)$.
Henceforth, $\mu_i$ is understood to have this scheme structure.

\begin{remark}
  We can use the splitting types of $E$ and $N_i$ to define locally closed subsets of $\widetilde H_{d,g}$ of higher codimensions.
  Doing so gives a stratification of $\widetilde H_{d,g}$ according to the isomorphism types of the bundles appearing in the Casnati--Ekedahl resolution.
  This stratification has a distinctly algebro-geometric favor, and it should be in some sense orthogonal to the topological stratification discussed in \autoref{rem:top}.
  See \cite{pat:15} for more on this stratification.
\end{remark}

\subsection{Relations between various divisor classes}
Assuming \autoref{conj:pic}, the divisors defined in \autoref{sec:ram} and \autoref{sec:splitting} ought to be expressible as linear combinations of the tautological divisors $\kappa$, $\zeta$, and $\delta$.
Such an expression for the higher syzygy divisors $\mu_i$ is the content of \autoref{sec:computation}.
In this section, we give the expressions for all the other divisors.

Denote by $E$ the Tschirnhausen bundle of the universal cover $\alpha \from {\orb C} \to {\orb P}$.
In addition to the divisors disused so far, it will be useful to also consider the following three auxiliary divisors:
\[ p_* \c_1(E)^2, \quad p_* \ch_2(E), \quad \pi_* \c_1(\omega_\alpha)^2.\]
Lastly, denote by $\lambda = \c_1 \left(\pi_* \omega_\pi\right)$ the class of the Hodge line bundle on $\widetilde H_{d,g}$ and by $K$ the canonical divisor class of $\widetilde H_{d,g}$.
Set 
\[ b = 2g+2d-2.\]
This is the degree of the branch divisor of the covers in $\widetilde H_{d,g}$.

\begin{proposition}\label{prop:relations}
  The following identities hold in $\Pic_\Q(\widetilde H_{d,g})$:
  \begin{enumerate}
  \item $12 \lambda = \kappa + \delta$
  \item $p_*\c_1(E)^2 = \frac{b}{2} \cdot \zeta$
  \item $p_*\ch_2(E) = \frac{1}{12} \cdot \kappa + \frac{1}{2} \cdot \zeta + \frac{1}{12} \cdot \delta$
  \item $\pi_* \c_1(\omega_\alpha)^2 = \kappa + 4\cdot \zeta$
  \item $T = 2\cdot\kappa + 6\cdot\zeta - \delta$
  \item $D = -3\cdot\kappa +(b-10)\zeta + \delta$
  \item $\mu = -\frac{d}{6}\cdot\kappa + \frac{b-2d}{2}\cdot \zeta + \frac{d}{6}\cdot\delta$
  \item $K = \kappa + \zeta - \delta$
  \end{enumerate}
\end{proposition}
\begin{proof}
  We compute all the divisor classes on a generic one parameter family $B \to \widetilde {\orb H}_{d,g}$.
  Let $\alpha \from C \to P$ be the pull-back of the universal family to $B$ with the two projections $\pi \from C \to B$ and $p \from P \to B$.
  Set $\sigma = -c_1(\omega_p)/2$.
  \begin{asparaenum}
  \item[(1)] This is the well-known Mumford relation.
  \item[(2)] Let $\beta \subset P$ be the branch divisor of $\alpha$.
    Since $\beta$ is cut out by a section of $(\det E)^{\otimes 2}$, we have 
    \[ [\beta] = 2\c_1(E).\]
    Since $p \from P \to B$ is a $\P^1$ bundle, we have a relation
    \[ [\beta] = a \sigma + p^* D\]
    for some $a \in \Z$ and $D \in \Pic(B)$.
    Since $[\beta]$ has degree $b$ on the fibers of $p$, we get $a = b$.
    By comparing $\sigma \cdot [\beta]$ and $[\beta]^2$, we get
    \begin{equation}\label{eqn:c12}
      \c_1(E)^2 = b \c_1(E) \cdot \sigma.
    \end{equation}
    Since $\beta$ is the push-forward of the ramification divisor of $\alpha$, which has class $c_1(\omega_\alpha)$, we have
    \[ \alpha_* \left(c_1(\omega_\alpha) \right) = 2\c_1(E).\]
    Multiplying the above by $\sigma$, noting that $\omega_\alpha \cdot \sigma = \omega_\pi \cdot \sigma$, and using \eqref{eqn:c12} yields the second relation.
  \item[(3)] Applying $Rpi_*$ to both sides of the equation
    \[ \alpha_* \O_C = \O_P \oplus E^\vee\]
    and using Grothendieck--Riemann--Roch for the right hand side yields the third relation.
  \item[(4)] Using $\c_1(\omega_\alpha) = c_1(\omega_\pi) + 2\sigma$ and \eqref{eqn:c12} yields the fourth relation.
  \item[(5, 6)]
    To get $T$ and $D$, we sketch the argument from \cite[Proposition~3.2]{pat:15}.
    Assuming $B$ is sufficiently generic, the only singularities of $\beta$ will be nodes and cusps, and the map from the ramification divisor $\rho$ to the branch divisor $\beta$ will be the normalization.
    A simple local computation of the branch divisor of a cover specializing to a point of $D$ or $T$ shows that the nodes correspond to intersections of $B$ with $D$, and the cusps with the intersections of $B$ with $T$.
    Therefore, we get
    \[ p_a(\beta) - p_a(\rho) = T + D.\]
    By adjunction on $C$ and $P$, this leads to 
    \[ (\beta^2 - 2\rho^2)/2 = T + D.\]
    The branch points of $\rho \to B$ correspond to the intersections of $B$ with $\delta$ or with $T$.
    From adjunction on $C$  and Riemann--Hurwitz, we get
    \[ 2\rho^2 + \beta \cdot c_1(\omega_\pi) = T+\delta.\]
    Solving for $T$ and $D$, and using the previous relations yields the fifth and the sixth relations.
  \item[(7)]The class of $\mu$ is given by the Bogomolov expression $\c_1(E)^2-2d\ch_2(E)$, which yields the seventh relation (See \autoref{sec:bog} for the Bogomolov expression).
  \item[(8)]We sketch two ways to compute the canonical divisor.
    Note that the map $\widetilde{\orb H}_{d,g} \to \widetilde H_{d,g}$ is unramified in codimension 1, 
    so the canonical class of the stack is the same as that of the coarse space.

    First, consider the morphism $\br \from U \to V$, where $V \subset \P(\Sym^b\P^1) \sslash \SL(2)$ is the open locus where at most two of the $b$ marked points coincide, $U \subset \widetilde H_{d,g}$ is the locus of covers where at most two branch points coincide, and $\br$ is the morphism that assigns to a cover its branch divisor.
    It is easy to check that the complements of $V$ and $U$ have codimension 2, and hence it suffices to work on $V$ and $U$ for divisor calculations.
    Let $\Delta \subset U$ be the complement of the locus of $b$ distinct points.
    A simple local calculation shows that
    \[ \br^{-1} \Delta = 3T + 2D + \delta.\]
    The canonical divisor of $U$ is 
    \[ K_U = -\frac{(b+1)}{2b-2} \cdot \Delta.\]
    By Riemann--Hurwitz, we get
    \[ K_W = \br^* K_U + 2T + D,\]
    which combined with the previous relations yield the eighth relation.

    Another way is to use the deformation theory of maps developed in \cite{ran:89}.
    We can identify the tangent space to $\widetilde {\orb H}_{d,g}$ at $\alpha \from C \to \P^1$ as the the kernel of the induced map
    \[ \Ext^1(\Omega_C, \O_C) \to \Ext^1(\Omega_{\P^1}, \alpha_* \O_C).\]
    We can compute the Chern classes of the bundles on $\widetilde{\orb H}_{d,g}$ defined by both terms, and their difference yields the Chern class of the tangent bundle of $\widetilde {\orb H}_{d,g}$.
  \end{asparaenum}
\end{proof}

\section{The generic splitting type}\label{sec:gen}
The goal of this section is to discuss the splitting type of the syzygy bundle $N_i$ for a generic cover, and to prove that it is balanced when $d$ divides $g-1$.
Note, however, that the degree of $N_i$ may be divisible by its rank even when $d$ does not divide $g-1$.
One may expect $N_i$ to be generically balanced even in this setting.
This is not quite true, as the following example shows for the first bundle $N_1$.

\begin{example}
  \label{example:imbalance}
  Consider a general degree $6$, genus $4$ cover $\alpha: C \to \P^{1}$.
  We will show that the splitting of $N_{1}$ is $\O_{\P^{1}}(2) \oplus \O_{\P^{1}}(3)^{\oplus 7} \oplus \O_{\P^{1}}(4)$. 
  The degree of $N_{1}$ is $27$, and its rank is $9$, so $N_{1}$ is balanced if and only if it has a summand of degree $\geq 4$. 

  Let $h$ denote the divisor class of the relative $\O(1)$ on $\P E$, and let $f$ be the class of a fiber of $\P E \to \P^1$.
  Then the linear system $|h-2f|$ restricts to the complete canonical system on $C \subset \P E$, and furthermore, every element of the linear system $|2h-4f|$ is obtained as a sum of products of elements in $|h-2f|$. Since the canonical model of $C$ lies on a unique quadric $Q$, we see that there is a unique element of $|2h-4f|$ containing $C$.  This, in turn, translates into an $\O(4)$ summand in $N_{1}$. 

\end{example}

The example above can be generalized, provided the genus is small compared to the degree.
For large $g$, however, we expect that all bundles in the Casnati--Ekedahl resolution will be balanced.
Evidence for this is given by the next theorem. 

\begin{theorem}[See \cite{buj.pat:15}]
  \label{theorem:genericF}
  The bundle $N_{1}$ is balanced for a general branched cover provided $g$ is much larger than $d$. 
  When $d$ divides $g-1$, all syzygy bundles $N_{i}$ are balanced for a general branched cover.
\end{theorem}
The statement for $N_1$ is the main result of \cite{buj.pat:15}; the statement for $d$ dividing $g-1$ is \cite[Proposition~2.4]{buj.pat:15}.

We  now give a brief overview of the proof that the syzygy bundles are generically balanced when this divisibility constraint holds.
Since the Hurwitz space is irreducible, and the condition of being balanced is open, it suffices to provide one example of a cover where it holds.

Consider the surface $S = E \times \P^{^1}$, where $E$ is any elliptic curve. Let $D$ be any smooth curve on $S$ with $D \cdot (\{e\} \times \P^1) = k$ and $D \cdot (E \times \{t\}) = d$.
We will argue that the projection $D \to \P^{1}$ has the property that every syzygy bundle $N_{i}$ is balanced.

The surface $S$ embeds in $\P^{d-1} \times \P^{1}$ so that the projection to $\P^{d-1}$ is the projection $S \to E$ composed with the embedding of $E$ as an elliptic normal curve of degree $d$. The curve $D$ is then the intersection of $S$ with a divisor $H \subset \P^{d-1} \times \P^{1}$which restricts to a hyperplane in every $\P^{d-1}$.  

The main point is that the minimal free resolution of the elliptic normal curve $E \subset \P^{d-1}$ (embedded by any complete linear system of degree $d$) has the same shape as the Casnati--Ekedahl resolution of a degree $d$ branched cover.
This is equivalent to saying that elliptic normal curves are arithmetically Gorenstein.
The minimal free resolution of $E \subset \P^{d-1}$ pulls back to a relative minimal free resolution of $\O_{S}$ as an $\O_{\P^{d-1} \times \P^{1}}$-module.
More precisely, we get a resolution 
  \begin{equation}\label{Sresolution}
    \begin{split}
      0 \to &\O_{\P^{d-1} \times \P^{1}}(-d) \to V_{d-3} \otimes \O_{\P^{d-1} \times
        \P^{1}}(-d+2) \to V_{d-4} \otimes \O_{\P^{d-1} \times
        \P^{1}}(-d+3) \to \cdots \\
      \cdots &\to 
      V_2 \otimes \O_{\P^{d-1} \times
        \P^{1}}(-3) \to 
      V_1 \otimes \O_{\P^{d-1} \times
        \P^{1}}(-2)^{\oplus r_{1}} \to \O_{\P^{d-1} \times \P^{1}} \to
      \O_{S} \to 0,
    \end{split}
  \end{equation}
  where the $V_{i}$ are vector spaces of the same dimension as the rank of the bundles $N_i$ in the Casnati--Ekedahl resolution of a degree $d$ branched cover, and the twists refer to twists by the pullback of $\O_{\P^{d-1}}(1)$. 
  The restriction of this resolution to the relative hyperplane $H$ yields the Casnati--Ekedahl resolution of $D = H \cap S$.
  Note that the pullback of $\O_{\P^{d-1}}(1)$ to $H$ is $\O_H(1) \otimes \pi^* L$ where $\pi \from H \to \P^1$ is the projection, and $L$ is a line bundle on $\P^1$.
  Therefore, the terms in the resolution \eqref{Sresolution} restrict to $\pi^* (V_i \otimes L^{-i-1}) \otimes \O_H(-i-1)$.
  We thus get $N_i = V_i \otimes L^{-i-1}$, which is balanced.

  Since $D$ is a curve of type $(d,k)$ on $E \times \P^{1}$, its genus $g$ is $d(k-1)+1$.
  This is where we get the degree-genus restriction $g \equiv 1 \pmod d$.

  \begin{remark}
  \label{remark:genusonefibrations}
  The strategy above required understanding the relative resolution of the (trivial) genus one fibration $S \to \P^1$.
  In general, if $f \from X \to \P^{1}$ is a genus one fibration with simple nodes as singularities, then a relative degree $d$ divisor $D \subset X$ yields a relative embedding 
  \begin{align*}
    X \hookrightarrow \P(f_{*}\O_{X}(D)) \to \P^{1}
  \end{align*}
  and $X$ enjoys a relative resolution with exactly the same form as the Casnati--Ekedahl resolution of a degree $d$  branched cover.
  The bundles appearing inthe relative resolution of $X$ and the Casnati--Ekedahl resolution for $D \to \P^{1}$ are determined by each other, and one is balanced if and only if the other is.
  In this way, the study of Casnati--Ekedahl resolutions is intimately related to the study of relative resolutions of genus one fibrations.  
\end{remark}

\begin{remark}
  \label{remark:normalization}
  One might be able to deduce that $N_{i}$ is as balanced as possible (that is, $h^1(\End E(-1)) = 0$) even when $g \not \equiv 1 \pmod d$ as follows.
  Notice that for a singular $D \subset E \times \P^{1}$, the argument sketched above still holds without change.
  If one understands how the syzygy bundles $N_{i}$ are related for $D$ and its normalization $\widetilde{D}$, one might be able to handle the cases where $g \not \equiv 1 \pmod d$.  
\end{remark}

The strategies outlined in \autoref{remark:genusonefibrations} and \autoref{remark:normalization} have not been fully explored.
The authors intend to investigate them in the future.
Notice that the idea of using branched covers on elliptic fibrations parallels the idea of using curves on K3 surfaces apr\'es \cite{laz:86}.

\section{The divisor class of $\mu_i$}\label{sec:computation}

The goal of this section is obtain the divisor class of the higher syzygy divisors $\mu_i$.

\subsection{The Bogomolov expression}\label{sec:bog}
Let $B$ be a smooth curve and $p \from P \to B$ a $\P^1$ bundle.
Let $E$ be a vector bundle of rank $r$ on $P$ which is balanced on the generic fiber of $p$.
Denote by $\mu(E)$ the locus of points in $B$ over which $E$ is unbalanced with the scheme structure given by the first Fitting ideal of $R^1p_* (\End E \otimes \O(-1))$.
\begin{proposition}\label{prop:bogomolov}
  In the above setup, we have
  \[ [\mu(E)] = \c_1^2(E) - 2r\ch_2(E)\]
\end{proposition}
\begin{proof}
  By definition, we have
  \[ [\mu(E)] = -\c_1 Rp_*(\End E \otimes \O(-1))\]
  By Grothendieck--Riemann--Roch, we get
  \begin{align*}
    \ch Rp_*(\End E \otimes \O(-1)) &= p_* \left(\ch(E) \otimes \ch(E^\vee) \ch \O(-1) \td(P/B) \right)\\
    &= 2r\ch_2(E) - \c_1^2(E).
  \end{align*}
\end{proof}
Let us call the expression $\c_1^2(E) - 2r\ch_2(E)$ the \emph{Bogomolov expression} and denote it by $\Bog(N_i)$.
Note that $\Bog(N_i) = \Bog(N_i \otimes L)$ for any line bundle $L$, which should be expected from the geometric interpretation.

\subsection{The Koszul resolution}
By \autoref{prop:bogomolov}, the problem of finding the divisor class of $\mu_i$ is reduced to finding $\c_1(N_i)$ and $\ch_2(N_i)$.
To calculate the Chern classes of the bundles $N_i$, we express them as cohomology bundles of a resolution involving more familiar bundles.
This is the Koszul resolution, which we now recall.

Let $R$ be a (Noetherian) ring and $E$ a locally free $R$-module of rank $r$.
Let $S = \Sym^*(E)$ be the symmetric algebra on $E$ and let $M$ be a graded $S$-module.
Suppose we have a graded resolution
\[ 0 \to F_k \to \dots \to F_1 \to F_0 \to M \to 0,\]
where
\[ F_i = \bigoplus_{j \geq 0} \ N_{ij} \otimes_R S(-i-j)\]
and the $N_{ij}$ are locally free $R$-modules.
Suppose the resolution is minimal in the sense that all the maps $F_{i+1} \to F_i$ have graded components in positive degree.
Then we have the identification
\begin{equation}\label{eqn:tor}
  N_{ij} = \Tor^i_S(M, R)_{i+j},
\end{equation}
where the subscript denotes the graded component.
The right hand side can be computed in another way.
Instead of using an $S$-resolution of $M$, we use the $S$-resolution of $R$ given by the Koszul complex
\[ 0 \to \wedge^r E \otimes_R S(-r) \to \dots \to \wedge^{p}E \otimes_R S(-p) \to \dots \to E \otimes_R S(-1) \to S \to R \to 0.\]
Tensoring by $M$ and taking the $(i+j)$th graded component yields the complex
\[K_{i+j}: \wedge^r E \otimes_R M_{i+j-r} \to \cdots \xrightarrow{d_{p-1}} \wedge^{p}E \otimes_R M_{i+j-p} \xrightarrow{d_p} \cdots \to E \otimes_R M_{i+j-1} \to M_{i+j}. \]
Let $H^p(K_{i+j}) = \ker d_p/\im d_{p-1}$ be the cohomology.
Then we get the identification
\[ \Tor_S^j(M, R)_{i+j} = H^i(K_{i+j}).\]
Combining with \eqref{eqn:tor}, we get
\[ N_{ij} = H^i(K_{i+j}).\]

Let us now turn to the Casnati--Ekedahl resolution of the universal finite cover $\alpha \from \orb C \to \orb P$.
Let $E = \ker(\omega_\phi \to \O_{\orb P})$ be the Tschirnhausen bundle and $\iota \from \orb C \to \P E$ the relative canonical embedding.
Let $I \subset S = \Sym^* E$ be the homogeneous ideal of $\orb C$.
The Koszul complex $K_{i+1}$ for the $S$-module $S/I$ is the following
\[K_{i+1}:  \wedge^{i+1}E \to \wedge^{i} E \otimes E \to \wedge^{i-1}E \otimes \alpha_*(\omega_\alpha^2) \to \dots \to \alpha_* (\omega_\alpha^{i+1}).\]
Denote by $K_{i+1}(j)$ the $j$th term in the above complex, starting from $j = 0$ and counting from the right to the left.

\begin{proposition}\label{prop:ch}
  Let $1 \leq i \leq d-3$ and let $N_i$ be the $i$th syzygy bundle of $\alpha$.
  Then we have
  \[ \ch N_i = \sum_{j=0}^{i+1} (-1)^{j-i-1} \ch \left(K_{i+1}(j)\right).\]
\end{proposition}
\begin{proof}
  From the Casnati--Ekedahl resolution of $\alpha$ and the identification of the syzygy bundles with the cohomology of the Koszul complex, we know that
  \[ H^p(K_{i+1}) =
  \begin{cases}
    N_i & \text{if $p = i$}\\
    0 & \text{otherwise}
  \end{cases}.
  \]
  Therefore, we have the equality
  \[ N_i = \sum_{j=0}^{i+1} (-1)^{j-i-1} K_{i+1}(j)\]
  in the K-ring, from which the formula for the Chern character follows.
\end{proof}

\subsection{The computation}
We now compute $\ch N_i$ using the expression in \autoref{prop:ch}.
Since we are ultimately only interested in $\c_1$ and $\ch_2$, we ignore all terms of degree higher than $2$.
We may assume, for example, that the computation is happening over a general curve $B \to \widetilde {\orb H}_{d,g}$.
Denote by $\pi \from C \to B$ and $p \from P \to B$ the two projections.

From \autoref{prop:ch}, we have
\begin{equation}\label{eqn:1}
  \begin{split}
  \ch N_i &=  \sum_{j=0}^{i+1} (-1)^{j-i-1} \ch \left(K_{i+1}(j)\right)\\
  &= \left(\sum_{j=0}^{i+1}(-1)^{j-1} \ch(\wedge^{i+1-j}E) \ch (\alpha_* \omega_\alpha^j)) \right) - \ch(\wedge^i E) + \ch(\wedge^{i+1}E) \ch(E^\vee).
\end{split}
\end{equation}
The two correction terms at the end are needed because the $j = 0$ and $j = 1$ terms in the summation are differ from the corresponding terms of the Koszul resolution in the following way (the computation is in the K-ring):
\begin{align*}
  [\wedge^{i+1}E] \otimes [\alpha_* \omega_\alpha^0] &= [\wedge^{i+1}E] \otimes [ \O + E^\vee] \\
  &= [K_{i+1}(i+1)] + [\wedge^{i+1}E]\otimes[E^\vee],\ \text{and}\\
  [\wedge^i E] \otimes [\alpha_* \omega_\alpha ] &= [\wedge^i E]\otimes[\O + E] \\
  &= K_{i+1}(i) + [\wedge^i E].
\end{align*}

Next, by Grothendieck--Riemann--Roch applied to $\alpha$ we get
\begin{equation}
  \label{eqn:omegas}
  \ch \alpha_* \omega_\alpha^\ell = \alpha_* \left(1 + \ell \cdot \c_1(\omega_\alpha) + \frac{\ell^2\c_1(\omega_\alpha)^2}{2}\right) \left(1 - \frac{\c_1(\omega_\alpha)^2}{2} + \frac{\c_1(\omega_\alpha)^2+c_2(\Omega_{C/P})}{12}\right).
\end{equation}
Note that $\c_1(\omega_\alpha)$ is the class of the ramification divisor of $\alpha \from C \to P$.
In particular, $\alpha_* \c_1(\omega_\alpha)$ is the class of the branch divisor, which is cut out by a section of $(\det E)^{\otimes 2}$.
Therefore, we get
\begin{equation}\label{eqn:br}
  \alpha_* \c_1(\omega_\alpha) = 2\c_1 E.
\end{equation}
Specializing \eqref{eqn:omegas} to the case $\ell = 0$ and comparing the degree two terms yields 
\begin{equation}\label{eqn:c2Omega}
   \ch_2 E = \alpha_*\left(\frac{\c_1(\omega_\alpha)^2+c_2(\Omega_{C/P})}{12}\right).
\end{equation}
After using \eqref{eqn:br} and \eqref{eqn:c2Omega} to simplify \eqref{eqn:omegas}, we get
\begin{equation}\label{eqn:omegaell}
  \ch \alpha_* \omega_\alpha^\ell = d + (2\ell-1)\c_1(E) + \left(\ch_2(E) + \frac{\ell^2+\ell}{2}\pi_*\c_1(\omega_\alpha)^2 \right).
\end{equation}

For a vector bundle $E$ of rank $d-1$, we have
\begin{align*}
  \ch_0 \wedge^\ell E &= {{d-1} \choose \ell}, \\
  \c_1(\wedge^\ell E) &= {{d-2} \choose {\ell - 1}} \c_1(E), \text{ and }\\
  \ch_2(\wedge^\ell E) &= {{d-2} \choose {\ell - 1}} \ch_2(E) + \frac{1}{2}{{d-3} \choose {l-2}} (\c_1(E)^2 - 2\ch_2(E)).
\end{align*}
Using these identities and using \eqref{eqn:omegaell}, we expand the terms $\ch(\wedge^{i+1-j}E)\ch(\alpha_*\omega_\alpha^j)$ and carry out the summation.
To evaluate the summation in a closed form, we use the following combinatorial identities\footnote{along with ample help from the computer algebra system \texttt{Maple} with its \texttt{sumtools} package}:
\begin{align*}
  \sum_{l=0}^p (-1)^l {a \choose {p-l}} &= {{a-1} \choose p}\\
  \sum_{l=0}^p (-1)^l {a \choose {p-l}} l &= - {{a-2} \choose {p-1}}\\
  \sum_{l=0}^p (-1)^l {a \choose {p-l}} l(l-1) &=  2 {{a-3} \choose {p-2}}.
\end{align*}
The result is the following:
\begin{equation}\label{eqn:chN}
  \begin{split}
    \ch_0 (N_i) &= \frac{i(d-2-i)}{d-1} {d \choose {i+1}} \\
    \c_1 (N_i) &= (d-2-i) {{d-2} \choose {i-1}} \c_1(E)\\
    \ch_2 (N_i) &= {{d-4} \choose {i-1}} \left(d \ch_2(E) +
      \frac{(d-4)i+2}{2(d-i-1)}\c_1^2(E) - \c_1(\omega_\alpha)^2
    \right)
  \end{split}
\end{equation}

We use this computation to finish a postponed proof from \autoref{prop:ni}.
\begin{corollary}\label{cor:degNi}
  $\deg N_i = (d-2-i) (g+d-1) {{d-2} \choose i-1}.$
\end{corollary}
\begin{proof}
  Follows from \eqref{eqn:chN} and that $\deg \c_1(E) = (g+d-1)$.
\end{proof}

\begin{theorem}
  \label{thm:bog}
  The push-forward to $\widetilde H_{d,g}$ of the Bogomolov expression for $N_i$ is the following linear combination the standard divisor classes: 
  \begin{align*}
    p_* \Bog(N_i) &= A_i \left(6(gd-6g+d+6) \cdot \zeta - d(d-12)\cdot \kappa - d^2 \cdot \delta \right),
  \end{align*}
  where the coefficient $A_i$ is given by
  \[ A_i = {{d-4} \choose {i-1}}^2 \frac{(d-2)(d-3)}{6(i+1)(d-i-1)}.\]
\end{theorem}
\begin{proof}
  This is a direct consequence of the results of the Chern class computation collected in \eqref{eqn:chN} and the relations in \autoref{prop:relations}.
\end{proof}
Note that $\Bog(N_i)$ is symmetric with respect to the change $i \leftrightarrow d-2-i$, consistent with the fact that $N_i$ and $N_{d-2-i}$ are isomorphic up to twisting and taking duals.

The main theorem (\autoref{thm:main}) follows from \autoref{prop:gen_balanced}, the interpretation of the Bogomolov expression (\autoref{sec:bog}), and \autoref{thm:bog}.

\subsection{The supports of $\mu_i$}\label{sec:supports}
Given that the divisor classes $[\mu_{i}]$ are proportional, it is natural to wonder if the divisors $\mu_i$ are supported on the same set.
It would be surprising if it were true, but we cannot yet preclude this. 

Some evidence towards this is provided by the work of Christian Bopp.
Using his Macaulay2 package \cite{bop.hof:}, he has found many examples where the jumping loci of syzygy bundles are not supported on the same set.
Although his examples are for higher codimension loci, we expect there to be examples also in the divisorial case.

\def\cprime{$'$}

\bibliographystyle{siam}

\end{document}